\documentclass[11pt]{amsart}
\usepackage{amsmath, amssymb, graphicx, xr, cite}
\usepackage[marginratio=1:1]{geometry}
\usepackage{bbm}
\usepackage[all,cmtip]{xy}
\usepackage{amscd}
\usepackage{tikz}

\graphicspath{{.}{figs/}}

\theoremstyle{plain}
\newtheorem{theorem}{Theorem}[section]

    \newtheoremstyle{TheoremNum}
        {\topsep}{\topsep}              
        {\itshape}                      
        {}                              
        {\bfseries}                     
        {.}                             
        { }                             
        {\thmname{#1}\thmnote{ \bfseries #3}}
    \theoremstyle{TheoremNum}
    \newtheorem{thmn}{Theorem}

\theoremstyle{definition}

\theoremstyle{definition}

\theoremstyle{definition}
\newtheorem{remark}[theorem]{Remark}

\theoremstyle{plain}
\newtheorem{proposition}[theorem]{Proposition}

    \newtheoremstyle{TheoremNum}
        {\topsep}{\topsep}              
        {\itshape}                      
        {}                              
        {\bfseries}                     
        {.}                             
        { }                             
        {\thmname{#1}\thmnote{ \bfseries #3}}
    \theoremstyle{TheoremNum}
    \newtheorem{prop}{Proposition}

\theoremstyle{plain}
\newtheorem{lemma}[theorem]{Lemma}

\theoremstyle{plain}

\theoremstyle{plain}

\theoremstyle{plain}

\theoremstyle{plain}

\DeclareMathOperator{\Hom}{Hom}
\DeclareMathOperator{\Mod}{Mod}
\newcommand{\tor}{\ensuremath{\operatorname{Tor}}}
\newcommand{\f}{\ensuremath{\mathbb{F}_2}}
\newcommand{\K}{\ensuremath{\mathbb{K}}}

\newcommand{\xelt}[2]{{}_{#2}x_{#1}}
\newcommand{\welt}[2]{{}_{#2}w_{#1}}
\newcommand{\ielt}[2]{{}_{#2}\iota_{#1}}
\newcommand{\yelt}[2]{{}_{#2}y_{#1}}

\begin{document}
\thispagestyle{empty}

\title{Sutured annular Khovanov homology and two periodic braids}

\author{James Cornish}
\address{Department of Mathematics, Columbia University,
 New York, NY 10027}
 \email{cornish@math.columbia.edu}

\date{Revised: \today}

\maketitle

\begin{abstract}
Let $\widetilde{K}$ be a two-periodic braid and let $K$ be its quotient.  In this paper we show there is a spectral sequence from the next-to-top winding number grading of the sutured annular Khovanov homology of the closure of $\widetilde{K}$ to the next-to-top winding number grading of the sutured annular Khovanov homology of the closure of $K$.
\end{abstract}

\section{Introduction}
\label{sec:intro}

In 2004 Asaeda-Przytycki-Sikora ~\cite{APS} introduced an invariant of links in $I$-bundles over surfaces $F$ which takes the form of a trigraded homology group that categorifies the Kauffman bracket skein module.  In the special case that the surface $F$ is an annulus this invariant is called \emph{sutured annular Khovanov homology} and is denoted $SKh$.  This case was further studied by Grigsby-Wehrli ~\cite{GW1,GW2}, Auroux-Grigsby-Wehrli ~\cite{AGW1,AGW2}, and Roberts ~\cite{Roberts}.

A link $L$ is said to have \emph{period} $r > 1$ if there is an auto-diffeomorphism of $S^3$ of order $r$ such that the fixed set is an unknot disjoint from $L$ and the diffeomorphism maps $L$ onto itself.  Note one can use the fixed set to place $L$ into a thickened annulus.  Murasugi found formulas that are satisfied by the Alexander polynomials ~\cite{Mur_alex} and the Jones polynomials ~\cite{Mur_jones} of periodic knots.  Hendricks ~\cite{Hendricks} categorified some of Murasugi's Alexander polynomial formulas via spectral sequences.

Lipshitz-Treumann ~\cite{Lipshitz} gave another approach to proving such categorified formulas using Hochschild homology.  They show that if $A$ is a dg algebra and $M$ a dg $A$-bimodule such that $A$ fulfills certain technical conditions there is a spectral sequence from the Hochschild homology of the derived tensor of $M$ with itself that converges to the Hochschild homology of $M$.

In this paper, we use that approach to prove:

\begin{thmn}[\ref{thm:main}]

Let $\sigma$ be a braid on $n+1$ strands, $\sigma^2$ denote its square as a braid, and  $\widehat{\sigma}, \widehat{\sigma^2}\subset A \times I$ denote the annular closures of $\sigma$ and $\sigma^2$ respectively.  Then there is a spectral sequence starting at $SKh(\widehat{\sigma^2}; n-1)$ that converges to $SKh(\widehat{\sigma}; n-1)$.

\end{thmn}

Note that here the $n-1$ in $SKh(\widehat{\sigma}; n-1)$ denotes the sutured annular Khovanov homology of $\widehat{\sigma}$ in the $n-1$ winding number grading, which measures the amount of looping around the central hole of the annulus.  This is an ungraded version of the theorem; for a more precise version see Section \ref{sec:main}.  One can also use recent work of Beliakova-Putyra-Wehrli ~\cite{BPW} to strengthen this theorem to apply to all periodic links; see Theorem \ref{thm:stronger}.

Let $q_{K,n-1}$ denote the Euler characteristic of $SKh(K; n-1)$ for a link $K$, which one can also compute via a skein relation.  Let $y$ be the variable that records the quantum grading.  Then the decategorification of Theorem \ref{thm:main} is:

\begin{prop}[\ref{prop:decat}] Let $\sigma$ be a braid on $n+1$ strands.  Then $y^{n-1}q_{\widehat{\sigma^2},n-1} \equiv q_{\widehat{\sigma},n-1}^2 \pmod 2$.
\end{prop}

Chbili ~\cite{Chbili} and Politarczyk ~\cite{Polit} studied the Khovanov homology of periodic links constructing what they call \emph{equivariant Khovanov homology} for such links, related to sutured annular Khovanov homology.  Seidel-Smith ~\cite{SS} also have a result, in this case for symplectic Khovanov homology, for periodic links.  They obtain a rank inequality~\cite[Corollary 3]{SS} for the symplectic Khovanov homology of a two-periodic link $K$ as compared to its quotient $\widetilde{K}$ which they obtain via a spectral sequence.

This paper is organized as follows. Section \ref{sec:back} is background on the definitions of the relevant algebras.  In Sections \ref{sec:sm_proper} and \ref{sec:pi} we show that these algebras fulfill the technical conditions outlined by Lipshitz-Treumann; in Section \ref{sec:sm_proper} we show that these algebras are homologically smooth and proper while in Section \ref{sec:pi} we go through the calculation necessary to show that they are $\pi$-formal.  We then use these conditions to prove Theorem \ref{thm:main} in Section \ref{sec:main}.  Finally, Section \ref{sec:dec_comp} contains the decategorification of Theorem \ref{thm:main} as well as an example computation illustrating Theorem \ref{thm:main}.

\section{Background}
\label{sec:back}
The algebra $A_n$ used in this paper is the quotient of a path algebra.  The graph $G_n$ of this path algebra is shown in Figure \ref{graphA}.  The ring the path algebra is defined over is $\K = \oplus_{v \in Vert(G)} \f$. Thus in the path algebra the vertices are idempotents, with elements of the algebra being strings of edges so that the head of one edge is the tail of the next.

\begin{figure}[h]
\centerline{
\xymatrix{
        		\bullet \ar@/^/[r]^{\xelt{0}{1}} \ar@/_/[r]_{\ielt{0}{1}} 
		& \bullet \ar@/^/[r]^{\xelt{1}{2}} \ar@/_/[r]_{\ielt{1}{2}}
		& \bullet \ar@/^/[r]^{\xelt{2}{3}} \ar@/_/[r]_{\ielt{2}{3}} & {\cdots} \ar@/^/[r]^{\xelt{n-2}{n-1}} \ar@/_/[r]_{\ielt{n-2}{n-1}}
		& \bullet \ar@/^/[r]^{\xelt{n-1}{n}} \ar@/_/[r]_{\ielt{n-1}{n}} 
		& \bullet }}
\caption{The graph $G_n$}
\label{graphA}
\end{figure}
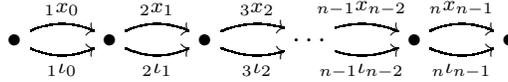


We quotient out by the relations $\ielt{i+1}{i+2}*\xelt{i}{i+1} = \xelt{i+1}{i+2}*\ielt{i}{i+1}$ and $ \xelt{i+1}{i+2}*\xelt{i}{i+1} = 0$ to form $A_n$ where multiplication should be thought of as concatenating paths starting from the rightmost element.  

A note about notation:  what we are denoting by $A_n$ is denoted by Auroux, Grigsby, and Wehrli in ~\cite{AGW1} and ~\cite{AGW2} as $B^{Kh}$, with an isomorphism that sends our $\xelt{i}{i+1}$ and $\ielt{i}{i+1}$ to $\xelt{i}{i+1}$ and ${}_{i+1}\mathbbm{1}_{i}$ respectively.  We then denote as $\bar{A}_n$ what is denoted in ~\cite{AGW1,AGW2} as $A_n$.  

The algebra $A_n$ has quadratic dual $B_n$.  It is also the quotient of a path algebra, the graph for which is shown in Figure \ref{graphB}.  The relations imposed are $\welt{i}{i-1}*\yelt{i+1}{i} = \yelt{i}{i-1}*\welt{i+1}{i}$ and $\yelt{i}{i-1}*\yelt{i+1}{i} = 0$.  To see $B_n$ is the quadratic dual of $A_n$, we think of $\welt{i+1}{i}$ as $\xelt{i}{i+1}^*$ and  $\yelt{i+1}{i}$  as $\ielt{i}{i+1}^*$; then the relations above for $B_n$ are dual to the relations for $A_n$ as required.    Note that $B_n$ is not directly related to $B$ in ~\cite{AGW1,AGW2}.

\begin{figure}[h]
\centerline{
\xymatrix{
        		\bullet & \ar@/^/[l]^{\welt{1}{0}} \ar@/_/[l]_{\yelt{1}{0}} 
		 \bullet & \ar@/^/[l]^{\welt{2}{1}} \ar@/_/[l]_{\yelt{2}{1}}
		 \bullet & \ar@/^/[l]^{\welt{3}{2}} \ar@/_/[l]_{\yelt{3}{2}} {\cdots} & \ar@/^/[l]^{\welt{n-1}{n-2}} \ar@/_/[l]_{\yelt{n-1}{n-2}}
		 \bullet & \ar@/^/[l]^{\welt{n}{n-1}} \ar@/_/[l]_{\yelt{n}{n-1}} 
		 \bullet }}
\caption{The graph for defining $B_n$}
\label{graphB}
\end{figure}
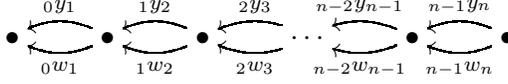

Let $\mathfrak{B}_{n+1}$ denote the braid group on $n+1$ strands.  To a braid $\sigma \in \mathfrak{B}_{n+1}$ we shall associate a bimodule $M_{\sigma}$ over $A_n$.  We first associate a module $M_{\sigma_i^{\pm}}$ to each elementary braid $\sigma_i^{\pm}$: if $\sigma = \sigma_{i_1}^{\pm}\cdots\sigma_{i_k}^{\pm}$ then $M_{\sigma} = M_{\sigma_{i_1}^{\pm}} \otimes_{A_n} \cdots \otimes_{A_n} M_{\sigma_{i_k}^{\pm}}$.  The modules $M_{\sigma_i^{\pm}}$ are defined as the mapping cones of two term complexes relating $A_n$ and projective modules that depend on $i$; the direction of the arrow depends on the sign of $\sigma_i^{\pm}$.  

Note once again that our $M_{\sigma}$ is denoted as $M_{\sigma}^{Kh}$ in ~\cite{AGW1,AGW2}.  We shall denote by $\bar{M}_{\sigma}$ what ~\cite{AGW1,AGW2} denote as $M_{\sigma}$.

Let $\sigma \in \mathfrak{B}_{n+1}$ and let $m(\widehat{\sigma}) \subset A \times I$ denote the mirror of the annular closure of $\sigma$.  Let $SKh(\widehat{\sigma}; n-1)$ denote the sutured annular Khovanov homology of $\widehat{\sigma}$ with winding number grading $n-1$, where the winding number grading appropriately measures the amount of wrapping around the central hole of the annulus.  We get the following lemma. Note $n_-$ and $n_+$ refer to the positive and negative crossings of $\widehat{\sigma}$ respectively, while the square brackets denote shifting the homological grading down and the curly brackets denote shifting the quantum grading up.

\begin{lemma}\label{lem:AGW} There is an isomorphism $SKh(\widehat{\sigma}; n-1) \cong HH(A_n, M_{m(\sigma)})[-n_-]\{(n-1)+n_{+}-2n_-\}$.

\end{lemma}

\begin{proof}
By ~\cite[Theorem 5.1]{AGW2}, we know that $SKh(\widehat{\sigma}; n-1) \cong HH(\bar{A}_n, \bar{M}_{m(\sigma)})[-n_-]\{(n-1)+n_{+}-2n_-\}$ as bigraded vector spaces.  In the proof of ~\cite[Theorem 6.1]{AGW2}, we see that $ HH(\bar{A}_n, \bar{M}_{m(\sigma)}) \cong  HH(A_n, M_{m(\sigma)})$. \end{proof}

\section{Homologically smooth and Proper}
\label{sec:sm_proper}

Two conditions that we need our algebras $A_n$ to satisfy are that they are homologically smooth and proper.  The fact that $A_n$ is proper is trivial.

\begin{proposition} The algebra $A_n$ is homologically proper for all $n$.
\label{prop:proper}
\end{proposition}
\begin{proof} Each algebra $A_n$ is itself finite. \end{proof}

However to see that $A_n$ is homologically smooth we shall first need a lemma.

\begin{lemma} Each $A_n$ is a homogeneous Koszul algebra.
\label{lem:koszul}
\end{lemma}

\begin{proof}
We will show that $A_n$ is a PBW algebra, then by ~\cite[Theorem 5.3]{Priddy} we will see that each $A_n$ is a homogeneous Koszul algebra.

First, note that $A_n$ is a homogeneous pre-Koszul algebra since it can be generated as an algebra by the elements of the form $ {}_{i}\iota_{i+1} $ and $ {}_{i}x_{i+1} $ with quadratic relations as above. These generators are also a Koszul basis: label the basis such that for any word in $A_n$  the $x$, if there is one, is on the far left, while ordering the basis such that $x$'s are greater than $\iota$'s.  Note that multiplying two elements of $A_n$ preserves this labeling, is zero, or we increase the lexicographical ordering when rearranging to get back to this labeling.  Moreover decomposing an element of $A_n$ as a product in any fashion also preserves this labeling.  Thus $A_n$ is indeed PBW and thus Koszul. \end{proof}

\begin{proposition} Each $A_n$ is homologically smooth.
\label{prop:smooth}
\end{proposition}

\begin{proof}

Recall that $B_n$ is the quadratic dual of $A_n$.  Thus, by  ~\cite[Proposition 3.1]{quadratic} and Lemma \ref{lem:koszul} we have that $\tor^{A_n}(\K, \K) \simeq B_n^*$, so the Koszul resolution of $A_n$ as found in ~\cite[Definition 3.7]{Priddy}  is  $A_n \otimes_{\K}  B_n^* \otimes_{\K} A_n $.  We claim that this is the required resolution by a finite cell retract.  The Koszul resolution is a resolution due to ~\cite[Theorem 3.8]{Priddy}; we claim that it is a finite cell retract of $A_n \otimes_{\f}  B_n^* \otimes_{\f} A_n$.  Since  $A_n \otimes_{\f}  B_n^* \otimes_{\f} A_n$ is finite and free it is certainly a finite cell bimodule.

To see that $A_n \otimes_{\K}  B_n^* \otimes_{\K} A_n$ is a cell retract of $A_n \otimes_{\f}  B_n^* \otimes_{\f} A_n$ note that the retract map is the natural map from setting more elements equivalent when going from tensoring over $\f$ to tensoring over $\K$.  To see the inclusion map, notice that for for any element $\alpha \in A_n \otimes_{\K}  B_n^* \otimes_{\K} A_n$ we have $\alpha = \sum_{i,j} a_i \otimes_{\K} {}_ib_j \otimes_{\K} {}_jc$ for some $a_i, {}_jc \in A_n, {}_ib_j \in B_n^*$ where $a_i k_i {}_ib_j = a_i {}_ib_j$ and ${}_ib_jk_j{}_jc = {}_ib_j{}_jc$ for some $k_i$ and $k_j$ that are primitive idempotents in $\K$. The inclusion map then sends $\alpha$ to $\sum_{i,j} a_i \otimes_{\f} {}_ib_j \otimes_{\f} {}_jc$.\end{proof}

\section{$\pi$-formality}
\label{sec:pi}

The final condition we need on our algebras $A_n$ is that they be $\pi$-formal, a condition that requires certain operations $d^{2i}$ vanish on Hochschild homology.  To define these operations let $M$ be an $A_n$ bimodule and let $R$ be a biprojective resolution of $A_n$.  Then the Hochschild chain complex $HC(M \otimes^L M)$, where $\otimes^L$ is the derived tensor product, is equal to $R \otimes_{A_n} (M \otimes_{A_n} R \otimes_{A_n} M)$ modded out by the equivalence relation generated by $r \otimes (m \otimes r' \otimes m'a) \sim ar \otimes (m \otimes r' \otimes m')$.  Let $\tau$ be the map on $HC(M \otimes^L M)$ that sends $r \otimes (m \otimes r' \otimes m')$ to $r' \otimes (m' \otimes r \otimes m)$.  By definition $\tau$ commutes with $\partial_{HC(M \otimes^L M)}$ and $\tau^2 = 1$, thus since we are working over $\f$ we have $(1+\tau)^2 = 0$.

We therefore define a bicomplex $HC_{*,*}^{Tate}(M\otimes^L M)$ where $HC_{p,q}^{Tate}(M\otimes^L M) = HC_q (M \otimes^L M)$ with the vertical differential being $\partial_{HC(M \otimes^L M)}$ and the horizontal differential being $(1+\tau)$.  We can induce a spectral sequence from this bicomplex by first looking at the horizontal filtration and then the vertical filtration.  Call the differential for this spectral sequence $d^r$.  In ~\cite[Proposition 3.10]{Lipshitz}, it is proved that $d^{2i+1}$ vanishes for all $i$.  Thus if $d^{2i}$ vanishes for $i \geq 2$ the spectral sequence would collapse at the $E^3$ page.  If the induced map $d^{2i}$ on $HH(M)$ vanishes for $i \geq 2$ for a bimodule $M$ we say that $M$ is \emph{$\pi$-formal}.

\begin{proposition}
Each $A_n$ is $\pi$-formal.
\label{prop:pi_formal}
\end{proposition}

\begin{proof}  We will drop the $n$ from $A_n$ and $B_n$ for this proof.  We need only check $d^{2i} $ vanishes on $ 1 \in \Hom (A^!, A^!)$ due to ~\cite[Theorem 5]{Lipshitz}.  Here $\Hom (A^!, A^!) \cong HH_o(A, A^!) \cong A \otimes B \otimes A \otimes B/\sim$ where $a_1 \otimes b_1 \otimes a_2 \otimes b_2 z \sim z a_1 \otimes b_1 \otimes a_2 \otimes b_2$.  Note we get this last isomorphism because the Koszul resolution $A \otimes B^* \otimes A$ gives us a model for $A^!$:

$$ A^! = \Hom_{{}_A\Mod_A}(A \otimes B^* \otimes A, A^e) = A \otimes B \otimes A.$$

Note that the differential on $A \otimes B \otimes A \otimes B/\sim$ is 

\begin{multline}
\partial(a_1 \otimes b_1 \otimes a_2 \otimes b_2) = \hspace{-5 em}\sum_{(\xi,\xi') \in \{( {}_i \iota_{i+1}, {}_{i+1}y_i), ( {}_i x_{i+1}, {}_{i+1}w_i) \}}\hspace{-5 em} a_1\xi \otimes \xi' b_1 \otimes a_2 \otimes b_2 + a_1 \otimes b_1\xi' \otimes \xi a_2 \otimes b_2 \\
+ a_1 \otimes b_1 \otimes a_2 \xi \otimes \xi' b_2 + \xi a_1 \otimes b_1 \otimes a_2 \otimes b_2 \xi'.
\end{multline}

For notational convenience let $x$,$y$ denote words in $A$, $B$ that include one $\xelt{i}{i+1}$ or $\yelt{i+1}{i}$ respectively, while $\iota$,$w$ denote words with only $\iota$'s or $w$'s.  Then when $(\xi,\xi'), (\eta, \eta') \in \{( \iota, y), ( x , w) \}$ with $\xi\eta$ and $\eta'\xi'$ both nonzero, we have $\xi\eta = x$ and $\eta'\xi' = y$.  We will call this the $xy$ phenomenon.  Note that since there are exactly two ways to get $\xi\eta = x$ and $\eta'\xi' = y$, when summed over all $(\xi,\xi'), (\eta, \eta')$ these terms vanish.

The rest of the proof is a concrete computation.  To keep notation under control, we will replace the tensor $\otimes$ with the vertical line $|$.  Remember that $\tau(a_1 \otimes b_1 \otimes a_2 \otimes b_2) = a_2 \otimes b_2 \otimes a_1 \otimes b_1$.  

We have that

\begin{align}
\partial (1|1|1|1) &= \sum_{(\xi,\xi')} (\xi|\xi'|1|1) + (1|\xi'|\xi |1) + (1|1|\xi|\xi') + (\xi|1|1|\xi') \nonumber \\
 &=(1+\tau)\Big( \sum_{(\xi,\xi')} (\xi|\xi'|1|1) + (1|\xi'|\xi |1) \Big) \label{eq1}
\end{align}

Let $(1+\tau)^{-1}$(\ref{eq1}) denote dropping the $(1+\tau)$ from Formula (\ref{eq1}).  Then

\begin{align}
\partial \circ (1+\tau)^{-1}(\ref{eq1}) = \sum_{(\xi,\xi'),(\eta,\eta')} & (\xi\eta|\eta'\xi'|1|1) + (\xi|\xi'\eta'|\eta|1) +(\xi|\xi'|\eta|\eta') +(\eta\xi|\xi'|1|\eta') +(\eta|\eta'\xi'|\xi |1) \label{eq2} \\
&+(1|\xi'\eta'|\eta\xi |1)+(1|\xi'|\xi\eta |\eta')+(\eta|\xi'|\xi |\eta') \nonumber
\end{align}

Note the first term when summed up over all $\xi,\eta$ is zero by the $xy$ phenomenon, as is the sixth.  The second and fifth terms cancel in the sum, while $\tau$ applied to the fourth term is the seventh over the sum of all $\xi,\eta$ .  Further:

\begin{align*}
\sum_{\xi,\eta} (\xi|\xi'|\eta|\eta') = \sum_i &(\xelt{i}{i+1}|\welt{i+1}{i}|\ielt{i}{i+1}|\yelt{i+1}{i}) + (\ielt{i}{i+1}|\yelt{i+1}{i}|\xelt{i}{i+1}|\welt{i+1}{i}) \\
&+ (\xelt{i}{i+1}|\welt{i+1}{i}|\xelt{i}{i+1}|\welt{i+1}{i}) + (\ielt{i}{i+1}|\yelt{i+1}{i}|\ielt{i}{i+1}|\yelt{i+1}{i})\\
\sum_{\xi,\eta} (\eta|\xi'|\xi |\eta') = \sum_i &(\xelt{i}{i+1}|\yelt{i+1}{i}|\ielt{i}{i+1}|\welt{i+1}{i}) + (\ielt{i}{i+1}|\welt{i+1}{i}|\xelt{i}{i+1}|\yelt{i+1}{i})\\
& + (\xelt{i}{i+1}|\welt{i+1}{i}|\xelt{i}{i+1}|\welt{i+1}{i}) + (\ielt{i}{i+1}|\yelt{i+1}{i}|\ielt{i}{i+1}|\yelt{i+1}{i})\\
\sum_{\xi,\eta} (\xi|\xi'|\eta|\eta') + (\eta|\xi'|\xi |\eta') = \sum_i &(\xelt{i}{i+1}|\welt{i+1}{i}|\ielt{i}{i+1}|\yelt{i+1}{i}) + (\ielt{i}{i+1}|\yelt{i+1}{i}|\xelt{i}{i+1}|\welt{i+1}{i}) \\
&+ (\xelt{i}{i+1}|\yelt{i+1}{i}|\ielt{i}{i+1}|\welt{i+1}{i}) + (\ielt{i}{i+1}|\welt{i+1}{i}|\xelt{i}{i+1}|\yelt{i+1}{i}) .
\end{align*}

Substituting, we have 
\begin{align}
(\ref{eq2}) = (1+\tau) \Big( \sum_{\xi,\eta} (\eta\xi|\xi'|1|\eta') + \sum_i (\xelt{i}{i+1}|\welt{i+1}{i}|\ielt{i}{i+1}|\yelt{i+1}{i})+(\xelt{i}{i+1}|\yelt{i+1}{i}|\ielt{i}{i+1}|\welt{i+1}{i}) \Big). \label{eq3}
\end{align}

Differentiating again, we have

\begin{align}
\partial \circ (1+\tau)^{-1}(\ref{eq3}) = &\sum_{\xi,\eta, \nu} (\eta\xi\nu|\nu'\xi'|1|\eta') + (\eta\xi|\xi'\nu'|\nu|\eta') +(\eta\xi|\xi'|\nu|\nu'\eta') +(\nu\eta\xi|\xi'|1|\eta'\nu') \label{eq4}\\
 &+\sum_{\nu}\sum_{i} ((\xelt{i}{i+1}\nu|\nu'\welt{i+1}{i}|\ielt{i}{i+1}|\yelt{i+1}{i})+(\xelt{i}{i+1}|\welt{i+1}{i}\nu'|\nu\ielt{i}{i+1}|\yelt{i+1}{i})\nonumber\\
&+(\xelt{i}{i+1}|\welt{i+1}{i}|\ielt{i}{i+1}\nu|\nu'\yelt{i+1}{i})+(\xelt{i}{i+1}|\yelt{i+1}{i}\nu'|\nu\ielt{i}{i+1}|\welt{i+1}{i})\nonumber\\
&+(\xelt{i}{i+1}|\yelt{i+1}{i}|\ielt{i}{i+1}\nu|\nu'\welt{i+1}{i})+(\nu \xelt{i}{i+1}|\yelt{i+1}{i}|\ielt{i}{i+1}|\welt{i+1}{i}\nu')),\nonumber
\end{align}
where a couple of terms, namely $(\nu x|w|\iota|y\nu')$ and $(x\nu|\nu'y|\iota|w)$, have been left off due to being zero; for example one of $\nu x$ or $y\nu'$ must be zero.

The first and fourth terms in Equation  \eqref{eq4} must cancel in the sum by the $xy$ phenomenon. For the rest of the sum there is exactly one $\iota$ or $w$ in each nonzero tensor after multiplying.  For example, in $ (\eta\xi|\xi'\nu'|\nu|\eta')$, if $\eta\xi = \iota$, then $\eta' = \xi' = y$, so to be nonzero we need $\nu' = w$, so $\nu = x$ and $ (\eta\xi|\xi'\nu'|\nu|\eta') = (w|y|x|y)$.  There is less cancellation than one might expect however because writing $(w|y|x|y)$ in the last sentence hides the fact that $w$ and the first $y$ are words of length two, while the $x$ and the second $y$ are words of length one.  In light of this, let $x_j$ denote the (non-unique) word that is of length $j$ and has one $\xelt{i}{i+1}$; similarly for $y_j$,$w_j$,$\iota_j$.

Nonetheless, since each tensor has exactly one $\iota$ or $w$, after differentiating each tensor will have zero $\iota$'s or $w$'s.  Thus even after factoring a $(1+\tau)$, the next differential will be zero.  But we shall finish the computation.

So we have

\begin{align*}
(\ref{eq4}) & = \sum_i (\iota_2|y_2|x_1|y_1)+(x_2|w_2|x_1|y_1)+(x_2|y_2|x_1|w_1)+(\iota_2|y_1|x_1|y_2)+(x_2|w_1|x_1|y_2)\\
&+(x_2|y_1|x_1|w_2)+(x_1|w_2|x_2|y_1)+(x_1|y_2|\iota_2|y_1)+(x_1|w_1|x_2|y_2)+(x_1|y_2|x_2|w_1)\\
&+(x_1|y_1|x_2|w_2)+(x_1|y_1|\iota_2|y_2)\\
&=(1+\tau)\Big( \sum_i (\iota_2|y_2|x_1|y_1)+(x_2|w_2|x_1|y_1)+(x_2|y_2|x_1|w_1)+(\iota_2|y_1|x_1|y_2)\\
&+(x_2|w_1|x_1|y_2)+(x_2|y_1|x_1|w_2)\Big), 
\end{align*}
where we left off the two terms that did cancel from Equation \ref{eq4}.

Differentiating, we have
\begin{align*}
\partial \circ (1+\tau)^{-1}(\ref{eq4}) &  = \sum_i ((x_3|y_3|x_1|y_1)+(x_3|y_2|x_1|y_2)+(x_3|y_3|x_1|y_1)+(x_2|y_3|x_2|y_1)\\
&+(x_2|y_2|x_2|y_2) +(x_3|y_2|x_1|y_2)+(x_3|y_2|x_1|y_2)+(x_3|y_1|x_1|y_3)\\
&+(x_3|y_2|x_1|y_2)+(x_2|y_2|x_2|y_2)+(x_2|y_1|x_2|y_3)+(x_3|y_1|x_1|y_3))\\
& = \sum_i (x_2|y_3|x_2|y_1) + (x_2|y_1|x_2|y_3).
\end{align*}

Thus
\begin{align}
\partial \circ (1+\tau)^{-1}(\ref{eq4}) & = (1+\tau)\Big( \sum_i (x_2|y_3|x_2|y_1) \Big). \label{eq5}
\end{align}

Finally, we then have
\begin{align}
\partial \circ (1+\tau)^{-1}(\ref{eq5}) & = 0. \qedhere
\end{align} \end{proof}

\section{Main Theorem}
\label{sec:main}

\begin {theorem}
\label{thm:main}

Let $\sigma \in \mathfrak{B}_{n+1}$, $\sigma^2$ denote its square in $\mathfrak{B}_{n+1}$, and  $m(\widehat{\sigma}), m(\widehat{\sigma^2}) \subset A \times I$ denote the mirror of the annular closures of $\sigma$ and $\sigma^2$ respectively.  Then there is a spectral sequence which has $E^1$ page isomorphic to $SKh(\widehat{\sigma^2}; n-1)$ and $E^\infty$ isomorphic to the associated graded algebra of a filtration of $SKh(\widehat{\sigma}; n-1)$.  If one shifts the quantum grading of $E^\infty$ up by $n-1$ and then divides the grading by two the isomorphism between $E^\infty$ and the associated graded algebra preserves the quantum grading.

\end{theorem}

\begin{proof}

By Lemma \ref{lem:AGW}, we know that $SKh(\widehat{\sigma^2}; n-1) \cong HH(A_n, M_{m(\sigma^2)})$ as bigraded vector spaces.  Now by Propositions \ref{prop:proper} and  \ref{prop:smooth}, our algebras $A_n$ are homologically smooth and proper.  Since $A_n$ is $\pi$-formal by Proposition \ref{prop:pi_formal}, by ~\cite[Theorem 5]{Lipshitz} we know that $M_{m(\sigma)}$ is $\pi$-formal.  Thus by ~\cite[Theorem 4]{Lipshitz}, there is a spectral sequence from  $HH(A_n, M_{m(\sigma^2)})$ to  $HH(A_n, M_{m(\sigma)})$.

Then by Lemma \ref{lem:AGW} again, we have that $HH(A_n, M_{m(\sigma)}) \cong SKh(\widehat{\sigma}; n-1)$.  Note that in the proof of ~\cite[Theorem 4]{Lipshitz} we see that the $E^{\infty}$ page is isomorphic to the Hochschild homology of $\sigma$ where every element has been tensored with itself and thus has had the quantum grading doubled.  Following this change and the grading shift of ~\cite[Theorem 5.1]{AGW2} one arrives at the grading shift given above. \end{proof}

In fact using recent work by Beliakova, Putyra, and Wehrli ~\cite{BPW} we can strengthen the ungraded version of the theorem to apply to all periodic links and obtain:

\begin{theorem}
\label{thm:stronger}
Let $T$ be an $(n+1,n+1)$ tangle, and  $\widehat{T}, \widehat{T^2} \subset A \times I$ denote the annular closures of $T$ and $T^2$ respectively.  Then there is a spectral sequence which has $E^1$ page isomorphic to $SKh(\widehat{T^2}; n-1)$ and $E^\infty$ isomorphic to the associated graded algebra of a filtration of $SKh(\widehat{T}; n-1)$.
\end{theorem}

\begin{proof}
In the proof of Lemma \ref{lem:AGW} use ~\cite[Theorem C]{BPW} instead of ~\cite[Theorem 5.1]{AGW2}.  You obtain the isomorphism

\begin{equation}
\label{eq7}
 SKh(\widehat{T}; n-1) \cong HH(A_n, M_{T}).
\end{equation}
Then follow the proof of Theorem \ref{thm:main} using Equation \eqref{eq7} instead of Lemma \ref{lem:AGW}.
\end{proof}

Note that one should be able to use ~\cite[Theorem C]{BPW} to extend this result to apply to all winding number gradings of $SKh$, not just the next-to-top one.  However one would have to redo Propositions \ref{prop:proper}, \ref{prop:smooth}, and \ref{prop:pi_formal} to apply to larger algebras.

\section{Decategorification and Computations}
\label{sec:dec_comp}

Theorem \ref{thm:main} has a decategorified statement.  Following Roberts in ~\cite{Roberts}, for a link $L$ in the thickened annulus let $q_{L,k} = \sum_{i,j} (-1)^i y^j \mathrm{rk}(SKh^{i,j}(L,k))$.  Then the graded Euler characteristic of sutured annular Khovanov homology is $\sum_k q_{L,k}x^k$; note that one can also define this from a skein relation as found in Section 2 of ~\cite{Roberts}.  In addition $\sum_k q_{L,k} $ is the Jones polynomial of the link in $S^3$.  Then we have the following proposition.

\begin{proposition}
\label{prop:decat}
 Let $\sigma$ be a braid on $n+1$ strands.  Then $y^{n-1}q_{\widehat{\sigma^2},n-1} \equiv q_{\widehat{\sigma},n-1}^2 \pmod 2$.
\end{proposition}

\begin{proof} Theorem \ref{thm:main} gives us a spectral sequence from the sutured annular Khovanov homology in winding number grading $n-1$ of $\widetilde{K}$ to that of $K$, and thus we get a relation on the graded Euler characteristics.  However, as above in Theorem \ref{thm:main} we need to shift the quantum grading for the spectral sequence; we therefore need to square the polynomial for $\widehat{\sigma}$ and multiply the polynomial for $\widehat{\sigma^2}$ by $y^{n-1}$ as in the statement of the proposition.  Note this relation is only true modulo 2 also due to grading issues; the spectral sequence in Theorem \ref{thm:main} is induced from a bicomplex where the horizontal differential has grading zero in $SKh^{i,j}(\widetilde{K},n-1).$  Thus the differentials in the spectral sequence do not respect the homological grading of $SKh^{i,j}(\widetilde{K},n-1)$.  However the spectral sequence works by canceling generators in pairs, thus the Euler characteristic is preserved modulo 2.   \end{proof}

We conclude with a detailed computation of an example.  Consider the Hopf link $\widetilde{L}$ inside the thickened annulus, thinking of the Hopf link as a braid on two strands with both strands being homologically nontrivial in the annulus.  The Hopf link is a two-periodic link whose quotient is the unknot $L$ with a single crossing; note $L$ is also homologically nontrivial in the annulus.  We have that $\widetilde{L}$ is the closure of $(\sigma_1^+)^2$ and $L$ is the closure of $\sigma_1^+$.  

By Lemma \ref{lem:AGW}, ~\cite[Theorem 4]{Lipshitz} and Proposition \ref{prop:smooth} , the zeroth page $E^0$ of the spectral sequence of Theorem \ref{thm:main} from $SKh(\widetilde{L}; 1)$ to $SKh(L; 1)$ is $M_{\sigma_1^-} \otimes_{A_1} A_1 \otimes_{\K}  B_1^* \otimes_{\K} A_1 \otimes_{A_1} M_{\sigma_1^-} \otimes_{A_1} A_1 \otimes_{\K}  B_1^* \otimes_{\K} A_1 / \sim$ which is isomorphic to $M_{\sigma_1^-} \otimes_{\K}  B_1^* \otimes_{\K} M_{\sigma_1^-} \otimes_{\K}  B_1^* \otimes_{\K} / \sim $, where note the second $\sim$ is the equivalence relation generated by $m_1 \otimes b_1^* \otimes m_2 \otimes b_2^* k_i \sim k_im_1 \otimes b_1^* \otimes m_2 \otimes b_2^*$ for  $k_i \in \K$.  Note that to have the first page $E^1$ isomorphic to $SKh(\widetilde{L}; 1)$ as bigraded vector spaces as in ~\cite[Theorem 5.1]{AGW2}  we need to shift $HH(A_1,M_{\sigma_1^-})$ by $[-n_-]\{(n-1)+n_{+}-2n_-\}$ where $n_-$ and $n_+$ refer to the positive and negative crossings in $\widetilde{L}$ respectively, the square brackets denote shifting the homological grading down, and the curly brackets denote shifting the quantum grading up.  We thus need to shift all quantum gradings up by two; in the below we have already done this shift.

Now for this spectral sequence $E^0$ has 34 elements; listing first the homological and then the quantum gradings, there are 7 elements in bigrading (0,2), 10 elements in bigrading (1,2), 4 elements in bigrading (2,2), 6 elements in bigrading (1,4), 6 elements in bigrading (2,4), and one element in bigrading (2,6).  We show this complex in the diagrams below.  Note the solid arrows are the differential in Hochschild homology while the dashed arrows denote the map $\tau$.  There are only two idempotents in $A_1$ and $B_1$; below the vertex at the tail of $\xelt{0}{1}$ is denoted $0$ while the vertex at the head of $\xelt{0}{1}$ is denoted $1$.  Since there are no more $x$ edges in $A_1$, $\xelt{0}{1}$ is denoted as $x$; similarly $\iota$, $y$, $w$ correspond to the obvious edges in $A_1$, $B_1^*$.  

Finally, note that in $M_{\sigma_1^-}$ there is a term that ~\cite{AGW1} call $P^{Kh}_1 \otimes {}_1 P^{Kh}$ with four elements $u^* \otimes u$, $v^* \otimes v$, $v^* \otimes u$, and $u^* \otimes v$ that are denoted below as $u$, $v$, $t$, and $s$ respectively.  Note that the bigradings of these elements, after a shift of $[-1]\{2\}$ in the $M_{\sigma_1^-}$ complex but before the {2} shift in the entire Hochschild complex, are $(1,0)$, $(1,0)$, $(2,1)$, and $(0, -1)$ respectively.  The bigradings of $0$, $1$, $\iota$, $x$, $w$, and $y$ are $(0,0)$, $(0,0)$, $(0,1)$, $(-1,-1)$, $(-2,-1)$ and $(-1,1)$ respectively before the {2} shift in the entire Hochschild complex.

First we look at the subcomplex of the $E^0$ page in quantum grading two.

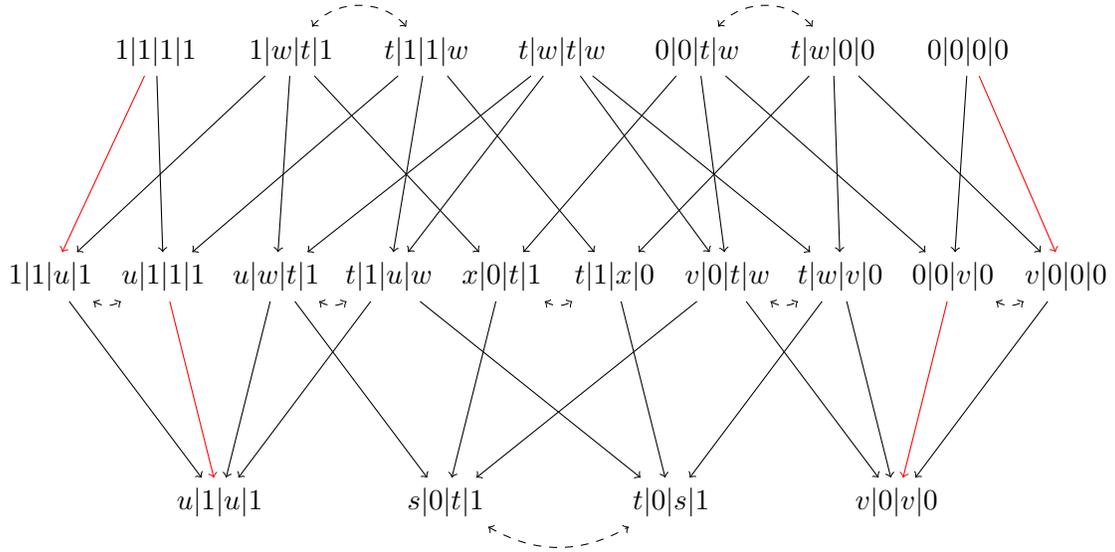
\begin{center}
\begin{tikzpicture}
  \node at (1.4,3) (m1) {$1|1|1|1$};
  \node at (3.2,3) (m2) {$1|w|t|1$};
  \node at (5,3) (m3) {$t|1|1|w$};
  \node at (6.8,3) (m4) {$t|w|t|w$};
  \node at (8.6,3) (m5) {$0|0|t|w$};
  \node at (10.4,3) (m6) {$t|w|0|0$};
  \node at (12.2,3) (m7) {$0|0|0|0$};
  \node at (0,0) (n1) {$1|1|u|1$};
  \node at (1.5,0) (n2) {$u|1|1|1$};
  \node at (3,0) (n3) {$u|w|t|1$};
  \node at (4.5,0) (n4) {$t|1|u|w$};
  \node at (6,0) (n5) {$x|0|t|1$}; 
  \node at (7.5,0) (n6) {$t|1|x|0$}; 
  \node at (9,0) (n7) {$v|0|t|w$};
  \node at (10.5,0) (n8) {$t|w|v|0$};
  \node at (12,0) (n9) {$0|0|v|0$}; 
  \node at (13.5,0) (n10) {$v|0|0|0$};
  \node at (2.25,-3) (o1) {$u|1|u|1$}; 
  \node at (5.25,-3) (o2) {$s|0|t|1$}; 
  \node at (8.25,-3) (o3) {$t|0|s|1$}; 
  \node at (11.25,-3) (o4) {$v|0|v|0$};
  \draw[->, red] (m1) to (n1);
  \draw[->] (m1) to (n2);
  \draw[->] (m2) to (n1);
  \draw[->] (m2) to (n5);
  \draw[->] (m2) to (n3);
  \draw[->] (m3) to (n2);
  \draw[->] (m3) to (n6);
  \draw[->] (m3) to (n4);
  \draw[->] (m4) to (n3);
  \draw[->] (m4) to (n4);
  \draw[->] (m4) to (n7);
  \draw[->] (m4) to (n8);
  \draw[->] (m5) to (n7);
  \draw[->] (m5) to (n9);
  \draw[->] (m5) to (n5);
  \draw[->] (m6) to (n8);
  \draw[->] (m6) to (n10);
  \draw[->] (m6) to (n6);
  \draw[->, red] (m7) to (n10);
  \draw[->] (m7) to (n9);
  \draw[->] (n1) to (o1);
  \draw[->, red] (n2) to (o1);
  \draw[->] (n5) to (o2);
  \draw[->] (n3) to (o1);
  \draw[->] (n3) to (o2);
  \draw[->] (n4) to (o1);
  \draw[->] (n4) to (o3);
  \draw[->] (n7) to (o4);
  \draw[->] (n7) to (o2);
  \draw[->] (n6) to (o3);
  \draw[->] (n8) to (o4);
  \draw[->] (n8) to (o3);
  \draw[->, red] (n9) to (o4);
  \draw[->] (n10) to (o4);
  \draw[<->, dashed, bend left=50] (m2) to (m3);
  \draw[<->, dashed, bend left=50] (m5) to (m6);
  \draw[<->, dashed, bend right=30] (n1) to (n2);
  \draw[<->, dashed, bend right=30] (n3) to (n4);
  \draw[<->, dashed, bend right=30] (n5) to (n6);
  \draw[<->, dashed, bend right=30] (n7) to (n8);
  \draw[<->, dashed, bend right=30] (n9) to (n10);
  \draw[<->, dashed, bend right=30] (o2) to (o3);
\end{tikzpicture}
\end{center}

To find the first page $E^1$ we cancel the vertical arrows; one way to do this is to cancel first the red arrows on the top row and then the red arrows on the bottom row.  One obtains the following complex:

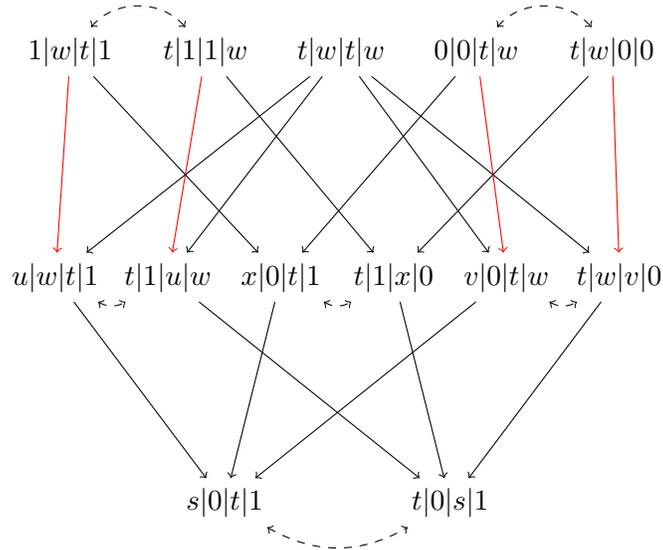
\begin{center}
\begin{tikzpicture}
  \node at (3.2,3) (m2) {$1|w|t|1$};
  \node at (5,3) (m3) {$t|1|1|w$};
  \node at (6.8,3) (m4) {$t|w|t|w$};
  \node at (8.6,3) (m5) {$0|0|t|w$};
  \node at (10.4,3) (m6) {$t|w|0|0$};
  \node at (3,0) (n3) {$u|w|t|1$};
  \node at (4.5,0) (n4) {$t|1|u|w$};
  \node at (6,0) (n5) {$x|0|t|1$}; 
  \node at (7.5,0) (n6) {$t|1|x|0$}; 
  \node at (9,0) (n7) {$v|0|t|w$};
  \node at (10.5,0) (n8) {$t|w|v|0$};
  \node at (5.25,-3) (o2) {$s|0|t|1$}; 
  \node at (8.25,-3) (o3) {$t|0|s|1$}; 
  \draw[->] (m2) to (n5);
  \draw[->, red] (m2) to (n3);
  \draw[->] (m3) to (n6);
  \draw[->, red] (m3) to (n4);
  \draw[->] (m4) to (n3);
  \draw[->] (m4) to (n4);
  \draw[->] (m4) to (n7);
  \draw[->] (m4) to (n8);
  \draw[->, red] (m5) to (n7);
  \draw[->] (m5) to (n5);
  \draw[->, red] (m6) to (n8);
  \draw[->] (m6) to (n6);
  \draw[->] (n5) to (o2);
  \draw[->] (n3) to (o2);
  \draw[->] (n4) to (o3);
  \draw[->] (n7) to (o2);
  \draw[->] (n6) to (o3);
  \draw[->] (n8) to (o3);
  \draw[<->, dashed, bend left=50] (m2) to (m3);
  \draw[<->, dashed, bend left=50] (m5) to (m6);
  \draw[<->, dashed, bend right=30] (n3) to (n4);
  \draw[<->, dashed, bend right=30] (n5) to (n6);
  \draw[<->, dashed, bend right=30] (n7) to (n8);
  \draw[<->, dashed, bend right=30] (o2) to (o3);
\end{tikzpicture}
\end{center}

We then cancel the red arrows in this complex to obtain:

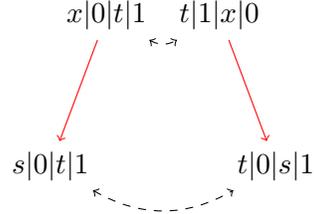
\begin{center}
\begin{tikzpicture}
  \node at (6.8,2) (m4) {$t|w|t|w$};
  \node at (6,0) (n5) {$x|0|t|1$}; 
  \node at (7.5,0) (n6) {$t|1|x|0$}; 
  \node at (5.25,-2) (o2) {$s|0|t|1$}; 
  \node at (8.25,-2) (o3) {$t|0|s|1$}; 
  \draw[->, red] (n5) to (o2);
  \draw[->, red] (n6) to (o3);
  \draw[<->, dashed, bend right=30] (n5) to (n6);
  \draw[<->, dashed, bend right=30] (o2) to (o3);
\end{tikzpicture}
\end{center}

One then cancels the final two red arrows to discover that only one element in bigrading $(0,2)$ survives to the first page from this subcomplex.  Below is the subcomplex of the $E^0$ page in quantum grading four.

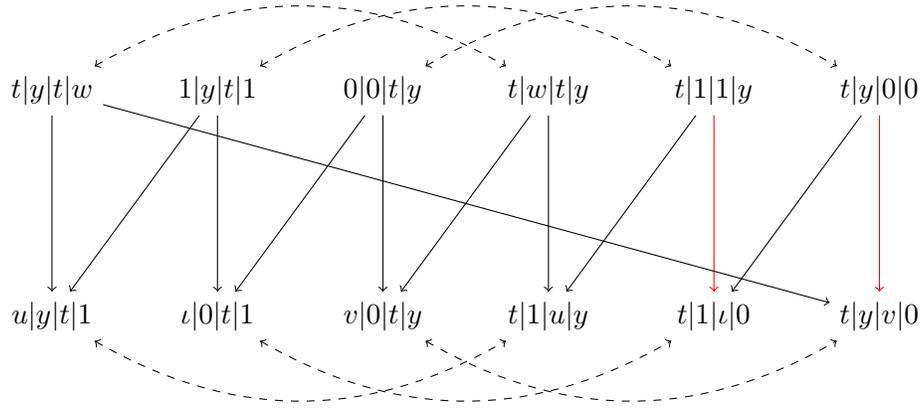
\begin{center}
\begin{tikzpicture}
  \node at (0,0)(a1) {$t|y|t|w$};
  \node at (2.2,0)(a2) {$1|y|t|1$};
  \node at (4.4,0)(a3) {$0|0|t|y$};
  \node at (6.6,0)(a4) {$t|w|t|y$};
  \node at (8.8,0)(a5) {$t|1|1|y$};
  \node at (11,0)(a6) {$t|y|0|0$};
  \node at (0,-3)(b1) {$u|y|t|1$};
  \node at (2.2,-3)(b2) {$\iota|0|t|1$};
  \node at (4.4,-3)(b3) {$v|0|t|y$};
  \node at (6.6,-3)(b4) {$t|1|u|y$};
  \node at (8.8,-3)(b5) {$t|1|\iota|0$};
  \node at (11,-3)(b6) {$t|y|v|0$};
  \draw[->, red] (a6) to (b6);
  \draw[->] (a6) to (b5);
  \draw[->] (a5) to (b4);
  \draw[->, red] (a5) to (b5);
  \draw[->] (a4) to (b4);
  \draw[->] (a4) to (b3);
  \draw[->] (a3) to (b3);
  \draw[->] (a3) to (b2);
  \draw[->] (a2) to (b2);
  \draw[->] (a2) to (b1);
  \draw[->] (a1) to (b1);
  \draw[->] (a1) to (b6);
  \draw[<->, dashed, bend left=30] (a1) to (a4);
  \draw[<->, dashed, bend left=30] (a2) to (a5);
  \draw[<->, dashed, bend left=30] (a3) to (a6);
  \draw[<->, dashed, bend right=30] (b1) to (b4);
  \draw[<->, dashed, bend right=30] (b2) to (b5);
  \draw[<->, dashed, bend right=30] (b3) to (b6);
\end{tikzpicture}
\end{center}

We cancel the two red arrows starting from the left to obtain the following complex:

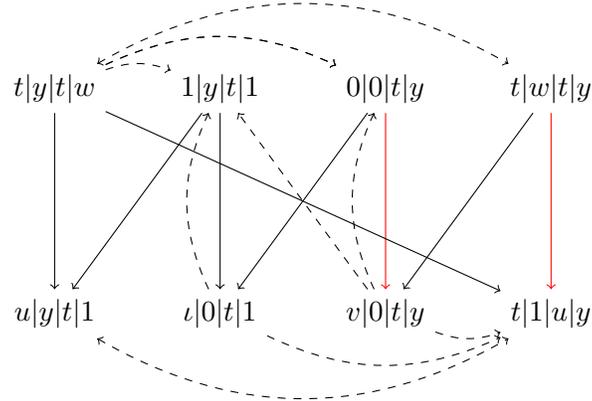
\begin{center}
\begin{tikzpicture}
  \node at (0,0)(a1) {$t|y|t|w$};
  \node at (2.2,0)(a2) {$1|y|t|1$};
  \node at (4.4,0)(a3) {$0|0|t|y$};
  \node at (6.6,0)(a4) {$t|w|t|y$};
  \node at (0,-3)(b1) {$u|y|t|1$};
  \node at (2.2,-3)(b2) {$\iota|0|t|1$};
  \node at (4.4,-3)(b3) {$v|0|t|y$};
  \node at (6.6,-3)(b4) {$t|1|u|y$};
  \draw[->, red] (a4) to (b4);
  \draw[->] (a4) to (b3);
  \draw[->, red] (a3) to (b3);
  \draw[->] (a3) to (b2);
  \draw[->] (a2) to (b2);
  \draw[->] (a2) to (b1);
  \draw[->] (a1) to (b1);
  \draw[->] (a1) to (b4);
  \draw[<->, dashed, bend left=30] (a1) to (a4);
  \draw[<->, dashed, bend right=30] (b1) to (b4);
  \draw[->, dashed, bend left=25] (b3) to (a3);
  \draw[->, dashed] (b3) to (a2);
  \draw[->, dashed, bend right=20] (b3) to (b4);
  \draw[->, dashed, bend right=25] (b2) to (b4);
  \draw[->, dashed, bend left=25] (b2) to (a2);
  \draw[->, dashed, bend left=25] (a1) to (a3);
  \draw[->, dashed, bend left=25] (a1) to (a3);
  \draw[->, dashed, bend left=20] (a1) to (a2);
\end{tikzpicture}
\end{center}

We again cancel the two red arrows starting from the left to obtain:

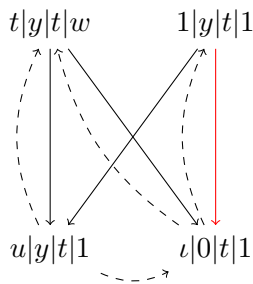
\begin{center}
\begin{tikzpicture}
  \node at (0,0)(a1) {$t|y|t|w$};
  \node at (2.2,0)(a2) {$1|y|t|1$};
  \node at (0,-3)(b1) {$u|y|t|1$};
  \node at (2.2,-3)(b2) {$\iota|0|t|1$};
  \draw[->, red] (a2) to (b2);
  \draw[->] (a2) to (b1);
  \draw[->] (a1) to (b1);
  \draw[->] (a1) to (b2);
  \draw[->, dashed, bend left=25] (b2) to (a2);
  \draw[->, dashed, bend left=25] (b1) to (a1);
  \draw[->, dashed, bend left=20] (b2) to (a1);
  \draw[->, dashed, bend right=25] (b1) to (b2);
\end{tikzpicture}
\end{center}

We finally cancel the red arrow to obtain:

\begin{center}
\begin{tikzpicture}
  \node at (0,0)(a1) {$t|y|t|w$};
  \node at (0,-3)(b1) {$u|y|t|1$};
  \draw[->, dashed] (b1) to (a1);
\end{tikzpicture}
\end{center}

Notice that this dashed arrow is therefore the only map that survives to the $E^1$ page; we already saw that the quantum grading two subcomplex was reduced to one element in the $E^1$ page while the quantum grading six subcomplex only had one element to begin with.  Thus the $E^1$ page has four elements, of bigradings $(0,2)$, $(1,4)$, $(2,4)$, and $(2,6)$, which is isomorphic to the sutured annular Khovanov homology of the Hopf link as it should be.

 Note that the dashed arrow above goes against the Hochschild homology grading by one while it goes with the $1+\tau$ grading by two; it is thus canceled at the $E^2$ page and the spectral sequence therefore stabilizes at the $E^3$ page.  Thus $E^{\infty}$ has two generators, of bigradings $(0,2)$ and $(2,6)$.  Notice that as in the proof of Proposition \ref{prop:decat} the isomorphism of this page to the sutured annular Khovanov homology of the unknot with one crossing doubles the quantum gradings.  We thus obtain a spectral sequence from $SKh(\widetilde{L}; 1)$ to $SKh(L; 1)$ as in Theorem \ref{thm:main}.

\begin{remark}
As stated in Section \ref{sec:intro} Seidel and Smith ~\cite{SS} also show there is a spectral sequence for two-periodic links to their quotients, although this spectral sequence is on symplectic Khovanov homology.  Note that, although it has not been shown that a quantum grading can be well defined for symplectic Khovanov homology, the spectral sequence of Seidel-Smith cannot preserve any quantum grading.  For the unlink with three components has period two with quotient the unlink with two components.  But notice the Jones polynomial of the unlink with three components is $q^3 + 3q + 3q^{-1} + q^{-3}$ while the Jones polynomial of the unlink with two components is $q^2+2+q^{-2}$ and these polynomials are not congruent modulo 2.
\end{remark}

\bibliography{Sutured_bib}{}
\bibliographystyle{plain}

\end{document}